\documentclass[a4paper,reqno,11pt,oneside]{amsart}
\usepackage{graphicx} 
\usepackage{amsmath}
\usepackage{amsthm}
\usepackage{amssymb}
\usepackage{color}
\usepackage{xcolor}
\usepackage{bbm}

\theoremstyle{plain}
\newtheorem{thm}{Theorem}
  \theoremstyle{plain}
  \newtheorem{lemma}[thm]{Lemma}
  \theoremstyle{plain}
  \newtheorem{corollary}[thm]{Corollary}
   \newtheorem{proposition}[thm]{Proposition}
  \theoremstyle{remark}
  \newtheorem{remark}[thm]{Remark}
  \newtheorem*{rem*}{Remark}
\def\div{\textup{div}}

\begin{document}
\title[On the simplicity of the sloshing eigenvalues  ]{On the simplicity of the sloshing  eigenvalues}
 
\author{Marco Ghimenti}
\address[Marco Ghimenti]{Dipartimento di Matematica
Universit\`a di Pisa
Largo Bruno Pontecorvo 5, I - 56127 Pisa, Italy}
\email{marco.ghimenti@unipi.it }

\author{Anna Maria Micheletti}
\address[Anna Maria Micheletti]{Dipartimento di Matematica
Universit\`a di Pisa
Largo Bruno Pontecorvo 5, I - 56127 Pisa, Italy}
\email{a.micheletti@dma.unipi.it }

\author{Angela Pistoia}
\address[Angela Pistoia] {Dipartimento SBAI, Universit\`{a} di Roma ``La Sapienza", via Antonio Scarpa 16, 00161 Roma, Italy}
\email{angela.pistoia@uniroma1.it}

\thanks{The first author is partially supported by the MIUR Excellence Department Project awarded to the Department of Mathematics, University of Pisa, CUP I57G22000700001 and by the PRIN 2022 project 2022R537CS \emph{$NO^3$ - Nodal Optimization, NOnlinear elliptic equations, NOnlocal geometric problems, with a focus on regularity}, founded by the European Union - Next Generation EU. The third author is partially supported by the MUR-PRIN-20227HX33Z ``Pattern formation in nonlinear phenomena''. 
The research of the authors is partially supported by the GNAMPA project 2024: ``Problemi di doppia
curvatura su variet\`a a bordo e legami con le EDP di tipo ellittico''}

\begin{abstract}
This paper investigates sloshing problems defined by $-\Delta u=0$ in $\Omega$, with mixed boundary conditions: $\partial_{\nu}u=\lambda u$ on $S$, and either $\partial_{\nu}u=0$ or $u=0$ on $W$. Here, $\Omega$ represents a smooth bounded domain in $\mathbb{R}^n$ with boundary $\partial\Omega=S \cup W$. We demonstrate that under small domain perturbations, all resulting eigenvalues are simple.
\end{abstract}

\keywords{ Sloshing problem, Steklov eigenvalues, generic properties, simplicity}

\subjclass{35J60, 35P05}
\maketitle

\section{Introduction}

We examine two problems involving Steklov eigenvalues with mixed boundary conditions, which are fundamentally linked to the {\em sloshing problem}—a classic topic in fluid mechanics. This problem describes the linear eigenvalue system governing the small oscillations of the free surface of an ideal fluid under gravity. Mathematically, this corresponds to finding the eigenvalues for the mixed Steklov-Neumann boundary value problem:
 \begin{equation}\label{Slosh0}  
   \left\{
    \begin{array}{cc}
         -\Delta u=0& \text{ on } \Omega\\
        \partial_\nu u=\lambda u & \text{ in } S\\
        \partial_\nu u=0 & \text{ in }  W
    \end{array}.
    \right.
\end{equation}
where $\Omega \subset \mathbb{R}^n$ ($n > 2$) is a connected bounded domain with boundary components $S$ (free surface) and $W$ (rigid walls). The eigenvalues $\lambda$ relate to the sloshing frequencies $\omega$ through $\lambda = \omega^2/g$. We also consider the case with a homogeneous Dirichlet condition on $W$ 
\begin{equation}\label{Slosh} 
   \left\{
    \begin{array}{cc}
         -\Delta u=0& \text{ on } \Omega\\
        \partial_\nu u=\lambda u & \text{ in } S\\
      u=0  & \text{ in }  W
    \end{array}.
    \right.
 \end{equation}

These problems also arise in the study of Cauchy processes (see \cite{baku}) and serve as models for stationary heat distribution where the boundary $W$ is either insulated  (Neumann condition) or maintained at zero temperature (Dirichlet condition), while the heat flux through $S$ is proportional to the temperature (Steklov condition). If $W = \emptyset$, the system simplifies to the classical Steklov eigenvalue problem
\begin{equation}\label{eq:steklov} 
   \left\{
    \begin{array}{cc}
         -\Delta u=0& \text{ on } \Omega\\
        \partial_\nu u =\lambda u & \text{ in } \partial\Omega\\
    \end{array}
    \right.
\end{equation}

It is worth noting that mixed Steklov-type eigenvalue problems have attracted considerable attention in recent years. We refer the reader to the insightful papers \cite{cggs,gp1,gipo}, which provide both a comprehensive collection of open problems and an exhaustive list of references on the subject.\\

It is well known (see \cite{agra, gipo}) that both \eqref{Slosh0} and \eqref{Slosh} possess a sequence of eigenvalues, counted with their multiplicities, denoted by$$  \lambda_1 < \lambda_2 \le \lambda_3 \le \dots$$It is straightforward to verify that the first eigenvalue is simple. For problem \eqref{Slosh0}, $\lambda_1 = 0$ and the associated eigenfunctions are the constant functions.
Regarding the classical sloshing problem \eqref{Slosh0}, it has been conjectured in \cite{kkm, gp1} that, in two dimensions, all eigenvalues are simple. In \cite{lpps}, the authors proved that for certain Lipschitz planar domains—specifically those with a flat free surface—all eigenvalues $\lambda_k$ are simple provided that $k$ is sufficiently large.While this conjecture remains open in the general case, we prove here that it is generically true. More precisely, we show that for any given domain $\Omega$, there exists an arbitrarily small perturbation such that all eigenvalues of the modified problems \eqref{Slosh0} and \eqref{Slosh} become simple. Our analysis specifically focuses on perturbations that leave one component of the boundary fixed.\\

The domain which we consider is an open connected bounded domain $\Omega$ in $\mathbb{R}^n$, whose boundary $\partial \Omega $ is $C^{1,1}$ and it is composed of two sets $S$ and $W$, with non empty $C^2$ disjoint internal parts  $\mathring S$ and $\mathring W$ with  $\mathring S\cap\mathring W=\emptyset$. \\
We introduce  a suitable Banach space of perturbations:
$$ \mathcal{D}:=\{\psi\in C^2(\mathbb{R}^n,\mathbb{R}^n),\ \|\psi\|:=\max_{i=0,1,2}\sup_{x\in\mathbb{R}^n}|\psi^{(i)}|<+\infty\}.$$
Given $\psi\in \mathcal{D}$  we define the perturbed domain by $\Omega_\psi := (I+\psi)(\Omega)$, and we rewrite  problem (\ref{Slosh0}) and  (\ref{Slosh}) on the domain $\Omega_\psi$.

Our main results are stated as follows.

\begin{thm}\label{th:SD}
For any $\varepsilon > 0$, there exists a perturbation $\psi \in \mathcal{D}$ with $\|\psi\| < \varepsilon$, leaving either $S$ or $W$ fixed, such that all eigenvalues of the Steklov-Dirichlet problem
\begin{equation}\label{S-Dpert}
\left\{
    \begin{array}{cc}
         -\Delta u=0& \text{ on } \Omega_\psi\\
        \partial_\nu u=\lambda u & \text{ on } (I+\psi)S=:S_\psi,\\
      u=0  & \text{ on }  (I+\psi)W=:W_\psi,
    \end{array}.
    \right.
\end{equation}
are simple.
 
\end{thm}

\begin{thm}\label{th:SN}
For any $\varepsilon > 0$, there exists a perturbation $\psi \in \mathcal{D}$ with $\|\psi\| < \varepsilon$ leaving $W$ fixed, such that all eigenvalues of the Steklov-Neumann  problem
\begin{equation}\label{S-Npert}
\left\{
    \begin{array}{cc}
         -\Delta u=0& \text{ on } \Omega_\psi\\
        \partial_\nu u=\lambda u & \text{ on } (I+\psi)S=:S_\psi,\\
    \partial_\nu  u=0  & \text{ on }  (I+\psi)W=:W_\psi,
    \end{array}.
    \right.
\end{equation}
are simple. In addition, if $\Omega \subset \mathbb{R}^2$, it is also possible to find a perturbation leaving $S$ fixed such that all eigenvalues of \eqref{S-Npert} are simple.
\end{thm}

Let us make some comments.

\begin{remark}
As a consequence of Theorem \ref{th:SD}, we deduce that all non-zero Steklov eigenvalues on a bounded domain are generically simple. More precisely, if $\Omega$ is a smooth bounded domain with $C^{1,1}$ boundary, for any $\varepsilon > 0$ there exists a perturbation $\psi \in \mathcal{D}$ with $\|\psi\| < \varepsilon$ such that all non-zero eigenvalues of the Steklov problem
\begin{equation}\label{eq:steklov-pert}
\left\{
    \begin{array}{cc}
         -\Delta u=0& \text{ on } \Omega_\psi\\
        \partial_\nu u=\lambda u & \text{ on } \partial \Omega_\psi,\\
           \end{array}.
    \right.
 \end{equation}
are simple.
The generic simplicity of Steklov eigenvalues was originally established by Wang \cite{wang} on compact Riemannian manifolds with boundary, following the method developed in \cite{u}. This result was recently extended by the same author to bounded Euclidean domains in a preprint \cite{wangpre}. Here, we recover the generic simplicity of Steklov eigenvalues on bounded domains using an entirely different technique.
\end{remark}

\begin{remark}
We note that these results imply that the eigenvalues of problems \eqref{Slosh0} and \eqref{Slosh} are simple for a generic domain $\Omega$. Indeed, for any given domain $\Omega$, there exists an arbitrarily close perturbed domain $\Omega_\psi = (I + \psi)\Omega$ such that all eigenvalues of the corresponding problem are simple. Furthermore, the perturbation $\psi$ can be chosen to act on only one component of the boundary. Notably, in our construction, the perturbations $\psi$ are chosen so as not to affect the interface $S \cap W$. It would be of particular interest to investigate the behavior of the eigenvalues when the perturbations are supported precisely on this interface, which remains a compelling direction for future research.
\end{remark}

\begin{remark}
We point out that for $\Omega \subset \mathbb{R}^n$, the statement of Theorem \ref{th:SN} could be refined: one can find a perturbation leaving $S$ fixed such that all eigenvalues of \eqref{S-Npert} have multiplicity at most $n-1$ (see Remark \ref{rem-n3}). It remains an open question whether this result is sharp.
\end{remark}

All the proofs  are based on an appropriate application of Micheletti's approach to simplicity of eigenvalues (see for example \cite{lupo,mi1,mi2}). The fundamental tool of this approach is the following abstract result (see \cite{mi2} and \cite[Section 3]{fgmp}) which provides a non splitting condition for eigenvalues under perturbation of the domain.

\begin{thm}\label{thm:astratto} 
Let $T_{b}:X\rightarrow X$ be a compact, self-adjoint operator on a Hilbert space $X$, which depending smoothly on a parameter $b$ in a real Banach space $B$. 
Assume that $T_{b}$ is Frechét differentiable in $b=0$. Let $\bar{\lambda}$ an eigenvalue for $T_0$ with multiplicity $m>1$  and let $x_{1}^{0},\dots,x_{m}^{0}$ form an orthonormal basisof the corresponding eigenspace. If  $\lambda(b)$ is an eigenvalue for $T_{b}$ such that $\lambda(0)=\bar\lambda$ and  $\lambda(b)$ maintains multiplicity $m$ for all $b$ with $\|b\|_{C^{0}}$ small, then for all such $b$ there exist a $\rho=\rho(b)\in\mathbb{R}$
for which
\begin{equation}
\left\langle T'(0)[b]x_{j}^{0},x_{i}^{0}\right\rangle _{X}=\rho\delta_{ij}\text{ for }i,j=1,\dots,m.
\label{eq:spezzamentoastratto}
\end{equation}
\end{thm}

Transversality results and their application to elliptic partial differential equation constitute a beautiful and fascinating topic. For readers wishing to explore this subject more deeply, we suggest, alongside the aforementioned papers, the fundamental references \cite{quinn, Saut-Temam,Smale,u2} 

The paper is organized as follows. The proof of Theorem \ref{th:SD} is given in full details through sections \ref{Sec:frame} to \ref{teorema1}. In particular in Section \ref{Sec:frame} we give the variational setting for the proof, in Section 
 \ref{Sec:der} we collect some technical computation which are fundamental in the proof and in Section \ref{nosplitti} we give the no splitting condition which comes from the application of Theorem \ref{thm:astratto} to our framework and which allows us to conclude the proof.

 Since the technique applies, which some changes, to Problem (\ref{Slosh0}), the proof of Theorem \ref{th:SN} is given in Section \ref{sec:sloshsimp}, where only the difference with the previous proof are underlined.

 Finally, in the appendix, we sketch a different apporach to Theorem  \ref{th:SN} , which can be interesting since add a variation to the main strategy which can be used in other situations.
 
\section{Variational framework for Theorem  \ref{th:SD}.}\label{Sec:frame}
We work in the Hilbert space $\mathcal{H}(\Omega):=\{u\in H^1(\Omega)\ :\ u=0 \text{ on }W\}$ equipped with the bilinear form $$a_\Omega(u,v):=\int_\Omega \nabla u \nabla v dx.$$
Since each $u\in \mathcal{H}$ vanishes on the portion $W\subset \partial \Omega$ of  positive surface measure, $a_\Omega(\cdot,\cdot)$ is an inner product equivalent to the standard one on $H^1(\Omega)$. We denote by $\|u\|_\mathcal{H} =\|u\|=(a_\Omega(u,u))^{1/2}$ the induced norm.

We say that $u$ is an eigenfunction with eigenvalue $\lambda$, that is $u$ solves \eqref{Slosh}  in a weak sense if and only if 
$$
    a_{\Omega}(u,\varphi)=\lambda\int_{\partial \Omega} u\varphi\ d\sigma .
$$
for all $\varphi\in \mathcal{H}(\Omega)$.

For any $f\in \mathcal{H}$, we consider the continuous linear functional on $\mathcal{H}$ defined by
$$
L_f(\varphi):=\int_{\partial\Omega} f\varphi\  d\sigma=\int_{S} f\varphi\  d\sigma
$$
By Riesz representation theorem, for any $f\in \mathcal{H}$, there exists a unique $w\in\mathcal{H}$, such that 
\[
a_{\Omega}(w,\varphi)=L_f (\varphi) ,
\]for all $\varphi\in \mathcal{H}$
We then define $E_\Omega$  the operator, from $ \mathcal{H}$ in itself, which associate $f$ to $w$ as above. Equivalently, if $w=E_\Omega (f)$, then 
\begin{equation}\label{eq:EuSD}
 a_{\Omega}(w,\varphi)=a_{\Omega}(E_\Omega(f),\varphi)=\int_{\partial \Omega} f\varphi\  d\sigma 
\text{ for all }\varphi \in \mathcal{H},
\end{equation}
 that is $w$ is a weak solution of  
\begin{equation}\label{S-Dw}
   \left\{
    \begin{array}{cc}
         -\Delta w=0& \text{ on } \Omega\\
        \partial_\nu w=f & \text{ in } S\\
        w=0 & \text{ in }  W
    \end{array}
    \right.
\end{equation}
So, in light of (\ref{S-Dw}),  $(\lambda, u)$, is an eigenpair of \eqref{Slosh}, in a weak sense, if and only if $u=E_\Omega(\lambda u)$, or equivalently, if $E_\Omega (u)=\frac 1\lambda u$.  In fact, if  $u=E_\Omega(\lambda u)$ it holds
\begin{equation}\label{identita}
     a_{\Omega}(u,\varphi)=a_{\Omega}(E_\Omega(\lambda u),\varphi)=\lambda\int_{\partial \Omega} u\varphi\  d\sigma 
\text{ for all }\varphi \in \mathcal{H}.
\end{equation}

As $\{\lambda_k\}_k$ is a sequence of positive increasing eigenvalues for problem  \eqref{Slosh}, the reciprocals  $\{\mu_k:=1/\lambda_k\}_k$  form a decreasing sequence of positive eigenvalues of $E_\Omega $ decreasing to zero while $k$ increases. Also, any $\mu_k$ admits the following Min-Max characterization:
\begin{eqnarray*}  
\mu_1= \max_{a_\Omega(\varphi,\varphi)=1}\int_{S}\varphi^2
&\quad&
\mu_t= 
\inf_{[v_1,\dots,v_{t-i}]}
\sup_{
\begin{array}{c}
a_\Omega(\varphi,v_i)=0\\
a_\Omega(\varphi,\varphi)=1
\end{array}}\int_{S} \varphi^2,
\text{ where } \varphi\in\mathcal{H}.
\end{eqnarray*}
In this way we transformed the original problem in an eigenvalue problem for the operator $E_\Omega$. 

This entire construction carries over to any perturbed domain $\Omega_\psi$. However, in order to use Theorem \ref{thm:astratto} in the proof,  it is necessary to work within the fixed space $\mathcal{H}(\Omega)$. To that end we pull back any function defined on $\Omega_\psi$ on the unperturbed domain $\Omega$ via the diffeomorphism $(I+\psi)$.  

Specifically, if $x\in \Omega $ and $\xi=(I+\psi)(x)=x+\psi(x)$, given a function $\tilde u\in \mathcal{H}(\Omega_\psi):=\{\tilde u\in H^1(\Omega_\psi)\ :\ \tilde u=0 \text{ on }(I+\psi)W\}$, we define the corresponding function $u\in \mathcal{H}(\Omega)$ by
 $$
\tilde u(\xi)=\tilde u(x+\psi (x)):=u(x).
$$
This defines a map 
\begin{align*}
\gamma_\psi&:\mathcal{H}(\Omega_{\psi})\rightarrow \mathcal{H}(\Omega);\\
\gamma_\psi(\tilde{u})&:=u(x)=\tilde{u}(x+\psi(x))=\tilde u\circ (I+\psi)(x).
\end{align*}
If $\|\psi\|_{C^2}<1/2$, the map $I+\psi$ is invertible with inverse $(I+\psi)^{-1}=I+\chi $ and the following maps are continuous isomorphisms:
\begin{eqnarray}
 \label{gamma-psi}
\gamma_{\psi} : \mathcal{H}(\Omega_{\psi})\rightarrow \mathcal{H}(\Omega)
& &
\gamma_{\psi}^{-1}=\gamma_{\chi}:\mathcal{H}(\Omega)\rightarrow \mathcal{H}(\Omega_{\psi}).
\end{eqnarray}
Obviously, it is possible to consider the same pull back and push forward maps defined in the whole spaces $H^1(\Omega)$  and $H^1(\Omega_\psi)$. We will use the same notation $\gamma_{\psi} $ and $\gamma_{\chi}$ in both cases. 
Now, given $\tilde u,\tilde \varphi \in \mathcal{H}(\Omega_\psi)$ we define a bilinear form $\mathcal{A}_\psi$  on $\mathcal{H}(\Omega)$ by pulling back the bilinear form  $a_{\Omega_\psi} $ on $\mathcal{H} (\Omega_\psi)$ as follows:
\begin{equation}
    \label{defA}
a_{\Omega_\psi}(\tilde u,\tilde \varphi)=\int_{\Omega_\psi} \nabla \tilde u \nabla \tilde \varphi \ dx  =:\mathcal{A}_\psi (u,\varphi) 
\end{equation}
where $u=\gamma_{\psi}\tilde u$ and $\varphi=\gamma_{\psi} \tilde \varphi$ .   
Recalling that $E_{\Omega_\psi}$ is introduced in (\ref{eq:EuSD}), we define
 \begin{eqnarray}\label{Tpsi}
T_\psi:\mathcal{H}(\Omega)\rightarrow\mathcal{H}(\Omega);&\quad& T_\psi{u}:=\gamma_\psi E_{\Omega_\psi} \gamma_\psi^{-1}{u}.
 \end{eqnarray}
Notice that $T_0=E_\Omega$. 
Recalling that  $\tilde{u},\tilde{\varphi}\in \mathcal{H}(\Omega_\psi)$,  ${u},{\varphi}\in \mathcal{H}(\Omega)$, and in light of  (\ref{eq:EuSD}) we have
  \begin{multline}
       \label{eq:SDfond}
     \mathcal{A}_\psi (T_\psi{u},\varphi) =a_{\Omega_\psi}(E_{\Omega_\psi}(\tilde u),\tilde \varphi)=\int_{\partial {\Omega_\psi}}  \tilde u \tilde \varphi\  d\sigma 
     =\int_{\partial \Omega} u\varphi\ B_\psi d\sigma 
 \    \text { for all }\varphi \in \mathcal{H}(\Omega). 
 \end{multline}
Here  $B_\psi$  represents the Jacobian  of the boundary change of variables. This final identity plays a crucial role in the application of the transversality theorem, allowing us to compute the derivatives of $T_\psi$ with respect to $\psi$
 
\section{Computation of derivatives}\label{Sec:der}

In this section we collect a series of technical Lemma necessary to compute the derivatives of the integral terms appearing in (\ref{eq:SDfond}), which will be the fundamental tool for our result.

\begin{lemma}\label{lem:Astorto} Set $$
\mathcal{A}_\psi(u,\varphi):
=\int_{\Omega_\psi} \nabla \tilde u\nabla \tilde \varphi d\xi.$$ 
Then, the derivative of $\mathcal{A}_\psi$ with respect to $\psi$ at $\psi=0$ satisfies
$$
\begin{aligned}
    \mathcal{A}'_\psi(0)[\psi](u,\varphi)=
&\int_\Omega-\psi_t \frac{\partial}{\partial x_t} \nabla u\nabla \varphi + \psi_i\frac{\partial }{ \partial x_j}\left(\frac{\partial u}{\partial x_i}\frac{\partial \varphi}{\partial x_j}+
\frac{\partial u}{\partial x_j}\frac{\partial \varphi}{\partial x_i}
\right) dx\\
&+\int_{\partial\Omega}\psi_t \nu_t \nabla u\nabla \varphi -
\psi_i \nu_j\left(\frac{\partial u}{\partial x_i}\frac{\partial \varphi}{\partial x_j}+\frac{\partial u}{\partial x_j}\frac{\partial \varphi}{\partial x_i}
\right) d\sigma
\end{aligned} $$
\end{lemma}
\begin{proof}
Observe that
\begin{displaymath}
    \tilde u(\xi)=(\gamma_\psi^{-1} u)(\xi)=(\gamma_\chi u)(\xi) =u(\xi +\chi \xi), 
\end{displaymath}
so 
$$
\frac{\partial }{\partial \xi_j}\tilde u=\frac{\partial }{\partial \xi_j}u(\xi +\chi \xi)=
\frac{\partial u}{\partial x_i}
\frac{\partial (\xi +\chi \xi)_i}{ \partial \xi_j}
=\frac{\partial u}{\partial x_j}+
\frac{\partial u}{\partial x_i}\frac{\partial \chi_i}{ \partial \xi_j}.
$$
At this point (with the Einstein convention on repeated indexes)
\begin{align*}
\mathcal{A}_\psi(u,\varphi):
=&
\int_{\Omega_\psi} \frac{\partial \tilde u }{\partial \xi_j}\frac{\partial \tilde \varphi }{\partial \xi_j}d\xi\\
=&\int_{\Omega_\psi} 
\left(
\frac{\partial u}{\partial x_j}+
\frac{\partial u}{\partial x_i}\frac{\partial \chi_i}{ \partial \xi_j}
\right)
\left(
\frac{\partial \varphi}{\partial x_j}+
\frac{\partial \varphi}{\partial x_s}\frac{\partial \chi_s}{ \partial \xi_j}
\right)\\
=&\int_\Omega \nabla u \nabla \varphi |J_\psi| dx+
\int_\Omega \frac{\partial u}{\partial x_j}
\frac{\partial \varphi}{\partial x_s}\frac{\partial \chi_s}{ \partial \xi_j}
|J_\psi| dx
+\int_\Omega \frac{\partial u}{\partial x_i}\frac{\partial \chi_i}{ \partial \xi_j}\frac{\partial \varphi}{\partial x_j}|J_\psi| dx\\
&+\int_\Omega \frac{\partial u}{\partial x_i}\frac{\partial \chi_i}{ \partial \xi_j}
\frac{\partial \varphi}{\partial x_s}\frac{\partial \chi_s}{ \partial \xi_j}
|J_\psi| dx\\
=&\int_\Omega \nabla u \nabla \varphi |J_\psi| dx+
\int_\Omega 
\left(\frac{\partial u}{\partial x_i}\frac{\partial \varphi}{\partial x_j}+
\frac{\partial u}{\partial x_j}\frac{\partial \varphi}{\partial x_i}
\right)\frac{\partial \chi_i}{ \partial \xi_j}
|J_\psi| dx\\
&+\int_\Omega \frac{\partial u}{\partial x_i}\frac{\partial \chi_i}{ \partial \xi_j}
\frac{\partial \varphi}{\partial x_s}\frac{\partial \chi_s}{ \partial \xi_j}
|J_\psi| dx
\end{align*}
Since $\|\psi\|_{C^2}<1/2$, we can express $I+\chi=(I+\psi)^{-1}=\sum_{t=0}^\infty (-1)^t(\psi)^t$, then, taking the linear part of the sum, we have the first derivatives of $\chi$, that is $\frac {\partial \chi_i}{\partial \xi_j} =-\frac {\partial \psi_i}{\partial x_j}$.
Also, by elementary computations
$$
|\mathrm{det}J_{I+\varepsilon \psi}|=1+\varepsilon \div\psi +o(\varepsilon)
$$
so $J_\psi'(0)[\psi])=\div\psi$ . With these equalities in mind, we can compute the derivatives of $\mathcal{A}_\psi(u,\varphi)$. In fact we have
\begin{eqnarray*}
    \mathcal{A}'_\psi(0)[\psi](u,\varphi)=&\int_\Omega\div\psi \nabla u\nabla \varphi - \left(\frac{\partial u}{\partial x_i}\frac{\partial \varphi}{\partial x_j}+
\frac{\partial u}{\partial x_j}\frac{\partial \varphi}{\partial x_i}
\right)\frac{\partial \psi_i}{ \partial x_j}dx
\end{eqnarray*} 
and integration by parts yields the desired result.
\end{proof}
\begin{remark}
    Notice that, when $\varphi$ is compactly supported in $\Omega$, then we simply have 
$$
\mathcal{A}'_\psi(0)[\psi](u,\varphi)=
\int_\Omega-\psi_t \frac{\partial}{\partial x_t} \nabla u\nabla \varphi + \psi_i\frac{\partial }{ \partial x_j}\left(\frac{\partial u}{\partial x_i}\frac{\partial \varphi}{\partial x_j}+
\frac{\partial u}{\partial x_j}\frac{\partial \varphi}{\partial x_i}
\right) dx
.$$
    
\end{remark}
\begin{lemma}\label{lem:B}
    Let $B_\psi$ denote the Jacobian of change of variables on $\partial \Omega$ induced by $\psi$. We have 
$$
\left. \frac{\partial}{\partial \psi}\int_{\partial \Omega} u\varphi\ B_\psi d\sigma\right|_{\psi=0}=\int_{\partial \Omega} u\varphi 
\left(
\div \psi -\sum_{r=1}^n 
\partial_\nu \psi_r \nu_r\right)d\sigma=
\left. \frac{\partial}{\partial \psi}\int_{\partial \Omega_\psi} \tilde u
\tilde \varphi\ d\sigma\right|_{\psi=0}
 $$
and
$$
\left. \frac{\partial}{\partial \psi}\int_{\partial \Omega}d(x)  u\varphi\ B_\psi d\sigma\right|_{\psi=0}=\int_{\partial \Omega} d(x)u\varphi
     \left(\div \psi -\sum_{r=1}^n \partial_\nu \psi_r \nu_r\right)d\sigma    
$$where $d:=1_S$
\end{lemma}

\begin{proof}
Given $\tilde u, \tilde \varphi\in H^1(\Omega_\psi)$, and set $\nu_{\Omega_\psi}$ the outward normal to $\partial \Omega_\psi$, we extend $\nu_{\Omega_\psi}$ smoothly to the whole $\Omega_\psi$, and we define the vectorial function $\tilde F:\Omega_\psi\rightarrow \mathbb{R}^n$ as $\tilde F:=\tilde u\tilde \varphi \nu_{\Omega_\psi}$.  Now, by the divergence theorem we have the identity
\begin{equation}\label{eq:derbordo0}
    \int_{\Omega_\psi}\frac{\partial}{\partial \xi_i}\tilde F_i d\xi=\int_{\partial \Omega_\psi}\tilde F_i(\nu_{\Omega_\psi})_i d\sigma =
    \int_{\partial \Omega_\psi}\tilde u(\xi) \tilde \varphi (\xi) d\sigma 
\end{equation}
By the usual change of variables we have $$\int_{\partial \Omega_\psi}\tilde u(\xi) \tilde \varphi (\xi) d\sigma= \int_{\partial \Omega} u(x)  \varphi (x) B_\psi d\sigma $$
Thus, differentiating in the $\psi$ variable at $\psi=0$ by  (\ref{eq:derbordo0})  we have 
\begin{equation}\label{eq:derbordo1}
    \left. \frac{\partial}{\partial \psi}\int_{\partial \Omega} u\varphi\ B_\psi d\sigma\right|_{\psi=0}=\left. \frac{\partial}{\partial \psi}
      \int_{\Omega_\psi}\frac{\partial}{\partial \xi_i}\tilde F_i d\xi
    \right|_{\psi=0}.
\end{equation}
Now, calling $F=\gamma_\psi \tilde F$, so $\tilde F (\xi)= \gamma_\psi^{-1} F(\xi)=\gamma_\chi F(\xi)=F(\xi +\chi \xi)$ and proceeding as in Lemma \ref{lem:Astorto}, we obtain 
\begin{equation}\label{eq:derbordo2}
 \int_{\Omega_\psi}\frac{\partial}{\partial \xi_i}\tilde F_i d\xi=
  \int_{\Omega}
 \left[
  \frac{\partial F_i}{\partial x_i}  -\frac{\partial  F_i }{\partial x_t}\frac{\partial \psi_t}{\partial x_i}
  \right] J_\psi dx
\end{equation}
Combining (\ref{eq:derbordo1}) and (\ref{eq:derbordo2}), and integrating by parts we get
$$\begin{aligned}
     \left. \frac{\partial}{\partial \psi}\int_{\partial \Omega} u\varphi\ B_\psi d\sigma\right|_{\psi=0}&=
      \left. \frac{\partial}{\partial \psi}
       \int_{\Omega}
       \left(\frac{\partial F_i}{\partial x_i}  -\frac{\partial  F_i }{\partial x_t}\frac{\partial \psi_t}{\partial x_i}\right) J_\psi dx\textbf{}  
        \right|_{\psi=0}\\ &
        =\int_\Omega \frac{\partial F_i}{\partial x_i}\div \psi dx-
        \int_\Omega \frac{\partial  F_i }{\partial x_t}\frac{\partial \psi_t}{\partial x_i}dx\\ &
        =
        -\int_\Omega F_i   \frac{\partial^2 \psi_t}{\partial x_i\partial x_t}dx
        +\int_{\partial\Omega} F_i\nu_i \div \psi d\sigma\\ &\quad
        +\int_\Omega  F_i \frac{\partial^2 \psi_t}{ \partial x_t\partial x_i}dx
      -  \int_{\partial \Omega } F_i \nu_t\frac{\partial \psi_t}{\partial x_i}d\sigma\\ &
      =\int_{\partial\Omega} u\varphi \div \psi d\sigma
      -  \int_{\partial \Omega } u\varphi \nu_i \nu_t\frac{\partial \psi_t}{\partial x_i}d\sigma
\end{aligned}$$
and, since $\nu_i \frac{\partial \psi_t}{\partial x_i}=\partial_\nu \psi_t$, we get the proof.

\end{proof} 
Finally, similarly to the proof of Lemma \ref{lem:Astorto}, we get the last result of this section. 
\begin{lemma}\label{lem:uv}
 We have 
$$\begin{aligned}
\left. \frac{\partial}{\partial \psi}\int_{ \Omega_\psi} \tilde u\tilde \varphi d\xi\right|_{\psi=0}&=
\left. \frac{\partial}{\partial \psi}\int_{ \Omega} u\varphi |J_\psi|dx\right|_{\psi=0}
\\ &=\int_\Omega-\psi_t \frac{\partial}{\partial x_t}  u \varphi  dx
+\int_{\partial\Omega}\psi_t \nu_t  u \varphi d\sigma
 \end{aligned}$$
\end{lemma}
To streamline notation, define:  
$$\begin{aligned}
    L_1(u,\varphi)&:=\int_\Omega-\psi_t \frac{\partial}{\partial x_t} \left(\nabla u\nabla \varphi\right) + \psi_i\frac{\partial }{ \partial x_j}\left(\frac{\partial u}{\partial x_i}\frac{\partial \varphi}{\partial x_j}+
\frac{\partial u}{\partial x_j}\frac{\partial \varphi}{\partial x_i}
\right) dx\\
  I_1(u,\varphi)&:=\int_{\partial\Omega}\psi_t \nu_t \nabla u\nabla \varphi -
\psi_i \nu_j\left(\frac{\partial u}{\partial x_i}\frac{\partial \varphi}{\partial x_j}+\frac{\partial u}{\partial x_j}\frac{\partial \varphi}{\partial x_i}
\right) d\sigma\\
  I_2( u,\varphi)&:=\int_{\partial \Omega} u\varphi \left(\div \psi -\sum_{r=1}^n \partial_\nu \psi_r \nu_r\right) d\sigma\\
   I_3( u,\varphi)&:= 
\int_\Omega-\psi_t \frac{\partial}{\partial x_t}  u \varphi  dx
+\int_{\partial\Omega}\psi_t \nu_t  u \varphi d\sigma
   \\
      G(u,\varphi)&:=\int_{ S} u\varphi \left(\div \psi -\sum_{r=1}^n \partial_\nu \psi_r \nu_r\right) d\sigma
 \end{aligned}$$

\section{No-splitting condition}
\label{nosplitti}
We aim to reformulate the abstract result, 
Theorem \ref{thm:astratto}, and in particular the no splitting condition (\ref{eq:spezzamentoastratto}), in our setting. 
Assume that $(e,\lambda)$  is an eigenpair for problem \eqref{Slosh}, and that the eigenvalue $\lambda$ has multiplicity $m>1$. Then $\mu=\frac1\lambda $ is an eigenvalue of $T_0$ (defined in \ref{Tpsi}) with multiplicity $m>1$. Let $\{e_1, \dots,e_m \}$ be an orthonormal base of the corresponding eigenspace, and suppose that any domain perturbation  $\psi$  preserves the multiplicity of $\mu$. Then, in light of  (\ref{eq:spezzamentoastratto}) we have to compute $   a_\Omega  (T_\psi '(0)[\psi] u,\varphi)$ and verify whether, for every such $\psi$, there exists $\rho=\rho(\psi)\in \mathbb{R}$ such that
$$
       a_\Omega  (T_\psi '(0)[\psi] e_r,e_s)=  \rho(\psi)\delta_{rs}.
$$

To compute the derivative $ a_\Omega  (T_\psi '(0)[\psi] e_r,e_s)$, we consider equation (\ref{eq:SDfond}), and, by virtue of the result of Section \ref{Sec:der}, we differentiate it with respect to $\psi$, obtaining the following lemma
\begin{lemma}\label{lem:diff}
The operators  $\mathcal{A}_\psi$ and $T_\psi$, respectively defined in (\ref{defA}) and (\ref{Tpsi})are Frechét differentiable 
with respect to $\psi$ at point $0$, and it holds
\begin{equation*}
    \mathcal{A}_\psi'(0)[\psi] (T_0 u,\varphi)+
      \mathcal{A}_0 (T_\psi '(0)[\psi] u,\varphi)=\left. \frac{\partial}{\partial \psi}\int_{S}  u\varphi\ B_\psi d\sigma\right|_{\psi=0}.
\end{equation*} 
\end{lemma}
At this point, recalling that $\mathcal{A}_0=a_\Omega$, by the previous lemma we have
\begin{equation}\label{cond1SD}\begin{aligned}
    a_\Omega  (T_\psi '(0)[\psi] u,\varphi)&=
    \left. \frac{\partial}{\partial \psi}\int_{S}  u\varphi\ B_\psi d\sigma\right|_{\psi=0}-
    \mathcal{A}_\psi'(0)[\psi] (T_0 u,\varphi)\\ &=
    -L_1(T_0u, \varphi)-
I_1( T_0u, \varphi)
+G(u,\varphi).
\end{aligned}\end{equation}

Having in mind (\ref{cond1SD}), and using that $T_0e_r=\mu e_r$ for all $r=1,\dots,m$, we obtain
\begin{equation}\label{cond2SD}
    -L_1(\mu e_r,e_s)-
I_1(\mu e_r,e_s)
+G( e_r,e_s)
=\rho\delta_{rs}.
\end{equation}
Now, consider a perturbation $\psi$ compactly supported in $\Omega$, so that $\psi\equiv0$ on the boundary.  In this case both $I_1$ and $G$ vanish, and condition (\ref{cond2SD}) reads as

$$\begin{aligned}
    L_1( \mu e_r,e_s)&=
  \mu \int_\Omega-\psi_i \frac{\partial}{\partial x_i} \nabla e_r\nabla e_s + \psi_i\frac{\partial }{ \partial x_j}\left(\frac{\partial e_r}{\partial x_i}\frac{\partial e_s}{\partial x_j}+
\frac{\partial e_r}{\partial x_j}\frac{\partial e_s}{\partial x_i}
\right) dx
\\ &=-\mu \delta_{rs}
\end{aligned}$$
Since this holds for arbitrary $\psi$, and  we may choose $\psi$ such that only one component $\psi_i$  is nonzero  at a time, then, for any $i$,  if $r\neq s$
\begin{equation*}
    \frac{\partial}{\partial x_i} \nabla e_r\nabla e_s +\sum_j\frac{\partial }{ \partial x_j}\left(\frac{\partial e_r}{\partial x_i}\frac{\partial e_s}{\partial x_j}+
\frac{\partial e_r}{\partial x_j}\frac{\partial e_s}{\partial x_i}
\right)=0
\end{equation*}
 almost everywhere on $\Omega$ and, when $r=s$, 
 \begin{equation*}
     \frac{\partial}{\partial x_i} |\nabla e_1|^2 +2\sum_j\frac{\partial }{ \partial x_j}\left(\frac{\partial e_1}{\partial x_i}\frac{\partial e_1}{\partial x_j}
\right)=\cdots=
\frac{\partial}{\partial x_i} |\nabla e_m|^2 +2\sum_j\frac{\partial }{ \partial x_j}\left(\frac{\partial e_m}{\partial x_i}\frac{\partial e_m}{\partial x_j}
\right)
 \end{equation*}
 That is means that, for any $\psi$ (not necessarily compactly supported in $\Omega$),  then $L_1(\mu e_r,e_s)=C(\psi)\delta_{rs}$, and   this term may be absorbed into the right-hand side of equation (\ref{cond2SD}). So the {\em no-splitting} condition becomes
 \begin{equation}\label{nosplitSD}
    -I_1(\mu e_r,e_s)
+G( e_r,e_s)=C(\psi)\delta_{rs}.
\end{equation}

\begin{remark}
To analyze the no-splitting condition \eqref{nosplitSD}, it is useful to recall that if an eigenfunction $e$ and its normal derivative $\partial_\nu e$ vanish in a neighborhood of a boundary point $\xi$ (where $\xi \notin \Gamma := S \cap W$), then $e$ must be identically zero.
Indeed, if $\Omega$ is a $C^{1,1}$ domain, the eigenfunction $e$ belongs to $H^2(D)$ for any subdomain $D \subset \Omega$ whose boundary is entirely contained within either $S$ or $W$ (see, e.g., \cite[Theorem 2.4.2.7]{grisvard}). It is well known that singular behavior occurs only near the points of the interface $S \cap W$ (see \cite{savare}). Furthermore, by the uniqueness of the Cauchy Problem as stated in \cite[p. 83]{henry}, we can deduce that $e$ vanishes identically in a neighborhood of $\xi$. Finally, by the Unique Continuation Principle, it follows that $e$ vanishes everywhere in $\Omega$.
\end{remark}

\subsection{Perturbations supported in $W$}\label{SDpertW}

If $(I+\psi)|_S=I|_S$, that is the support of  $\psi\cap\partial \Omega$ is contained in $W$,  then 
$G(e_r,e_s)\equiv 0$ and the boundary integrals are restricted to $W$ . So the no splitting condition (\ref{nosplitSD}) now becomes
  \begin{equation}\label{WSD}
-\mu \int_{W}\psi_t \nu_t \nabla e_r\nabla e_sd\sigma +\mu \int_{W}
\psi_i \nu_j\left(\frac{\partial e_r}{\partial x_i}\frac{\partial e_s}{\partial x_j}+\frac{\partial e_r}{\partial x_j}\frac{\partial e_s}{\partial x_i}
\right) d\sigma=C(\psi)\delta_{rs}.
\end{equation}
On $W$, since $e_r\equiv e_s\equiv0$ , it follows that $ \nabla e_r \nabla e_s=\frac{\partial e_r}{\partial \nu }\frac{\partial e_s}{\partial \nu }$. Furthermore, $\nu_j\frac{\partial e_r}{\partial x_j}=\frac{\partial e_r}{\partial\nu}$, 
and the same holds for $e_s$. Let us now choose $\psi=\alpha(x)\nu(t)$, where $\alpha(x)$ is an arbitrary smooth real function compactly supported in $W$. Thus  (\ref{WSD}) becomes
$$\begin{aligned}
-\mu\int_{W}\alpha \frac{\partial e_r}{\partial \nu} 
\frac{\partial e_s}{\partial \nu}d\sigma
+\mu\int_{W}
\alpha \nu_i\left(\frac{\partial e_r}{\partial x_i}\frac{\partial e_s}{\partial \nu}+\frac{\partial e_r}{\partial \nu}\frac{\partial e_s}{\partial x_i}
\right) d\sigma&
=\mu\int_{W}\alpha \frac{\partial e_r}{\partial \nu} 
\frac{\partial e_s}{\partial \nu}
\\ &=C(\psi)\delta_{rs}.
\end{aligned}$$
Since $\alpha$ is arbitrary, it must be that $\frac{\partial e_r}{\partial \nu}\frac{\partial e_s}{\partial \nu}=0$ if $r\neq s$ and $\left(\frac{\partial e_1}{\partial \nu}\right)^2=\dots=\left(\frac{\partial e_m}{\partial \nu}\right)^2$, that implies $\frac{\partial e_1}{\partial \nu}=\dots=\frac{\partial e_m }{\partial \nu}=0$.  Since we have also that $e_1=\dots=e_m=0$, the unique continuation property implies that the eigenfunctions must vanish identically, yielding a contradiction.

\subsection{Perturbations supported in $S$}\label{SDpertS}

Now, we take a perturbation which leaves $W$ fixed, so which is supported near $S$. 
The no splitting condition (\ref{nosplitSD}) is 
 \begin{equation} \begin{aligned}\label{splitS-SD}
    C(\psi)\delta_{rs}&=
    -I_1(\mu e_r,e_s)
+G( e_r,e_s) \\ &=
-\mu \int_{S}\psi_t\nu_t \nabla e_r\nabla e_s d\sigma
    +\mu \int_{S}
\psi_i\nu_j\left(\frac{\partial e_r}{\partial x_i}\frac{\partial e_s}{\partial x_j}+\frac{\partial e_r}{\partial x_j}\frac{\partial e_s}{\partial x_i}
\right) d\sigma\\ &\quad
+ \int_{ S}  e_r  e_s \div \psi d\sigma - \int_{ S}  e_r  e_s\sum_{r=1}^n \partial_\nu \psi_r \nu_r d\sigma
 \end{aligned}\end{equation}

Now, since $\partial_\nu e_r=\lambda e_r=1/\mu e_r$, on $S$ (and the same holds for $e_s$)  we have $\left(\frac{\partial e_r}{\partial \nu}\frac{\partial e_s}{\partial x_i}+\frac{\partial e_r}{\partial x_i}\frac{\partial e_s}{\partial \nu}
\right)=\frac1\mu\left( e_r\frac{\partial e_s}{\partial x_i}+\frac{\partial e_r}{\partial x_i} e_s
\right)=\frac1\mu\frac{\partial }{\partial x_i}(e_r e_s)$, so, integrating by parts, 
we have
$$ \begin{aligned}
\mu \int_{S}
\psi_i\nu_j\left(\frac{\partial e_r}{\partial x_i}\frac{\partial e_s}{\partial x_j}+\frac{\partial e_r}{\partial x_j}\frac{\partial e_s}{\partial x_i}
\right) d\sigma  &
=\mu \int_{S}
\psi_i\left(\frac{\partial e_r}{\partial x_i}\frac{\partial e_s}{\partial \nu}+\frac{\partial e_r}{\partial \nu}\frac{\partial e_s}{\partial x_i}
\right) d\sigma\\ &
=\int_{S}  
\psi_i \frac{\partial }{\partial x_i}(e_r e_s)
d\sigma\\ &=-\int_{S}  
\div\psi_i  (e_r e_s)
d\sigma
 \end{aligned}$$
Thus (\ref{splitS-SD}) becomes
$$
     C(\psi)\delta_{rs}=-\mu \int_{S}\psi_t\nu_t \nabla e_r\nabla e_s d\sigma
     - \int_{ S}  e_r  e_s\sum_{r=1}^n \partial_\nu \psi_r \nu_r d\sigma
$$
We can choose $\psi$ arbitrarily defined on $S$ and such that $\frac{\partial \psi_r} {\partial \nu}=0$ on $S$ for all $r=1,\dots,n$. Thus we get that  $\nabla e_r \nabla e_s=0$ if $r\neq s$ and $|\nabla e_1|^2=\dots=|\nabla e_m|^2$ almost everywhere on $S$. Thus the term $\int_{S}\psi_t\nu_t \nabla e_r\nabla e_s d\sigma$ can be absorbed in the left hand side of the condition and we remain with
$$
     C(\psi)\delta_{rs}=
     - \int_{ S}  e_r  e_s\sum_{r=1}^n \partial_\nu \psi_r \nu_r d\sigma
.$$
Since we can choose $\sum_{r=1}^n \partial_\nu \psi_r \nu_r $ as and arbitrary function on $S$, we get that $ e_r  e_s=0$ if $r\neq s$ and $| e_1|^2=\dots=| e_m|^2$ almost everywhere on $S$. This implies that on $S$ we get, for any $r=1,\dots, m$
$$
    e_r=0\ \hbox{and}\  \frac{\partial e_r}{\partial \nu}=\lambda e_r=0
,$$
which leads to a contradiction.

\section{Proof of Theorem \ref{th:SD}}\label{teorema1}
In this section we give the proof of Theorem \ref{th:SD}. While we outline the main steps of the argument, some technical details are omitted for the sake of readability.  For a fully detailed version of the proof, we refer the reader to \cite[Proof of Thm 1]{fgmp},\cite[Sec 5]{MiSNS}, and \cite[Lemma 7]{mi2}, where the same strategy is employed and all arguments are provided in full. The proofs of the other theorems in this manuscript follow an entirely analogous scheme. Therefore in the forthcoming sections  we restrict our attention to verifying the no splitting condition, which is the key ingredient required to implement this approach.
 
We start proving that we can split a single eigenvalue of multiplicity $m$ in several distinct  eigenvalues each with multiplicity strictly less then $m$. Then, by iterating the procedure, we can find a perturbation for which this eigenvalue splits in $m$ simple eigenvalue, and finally prove the main theorem of this paper.
We recall that if ${\lambda}$ an eigenvalue for the problem \eqref{Slosh} , then $\mu=1/\lambda$ an eigenvalue for the operator $T_0=E_\Omega$, and that the perturbed operator $T_\psi$ is defined in (\ref{Tpsi}). Finally, we will call $\mu^{\Omega_\psi}$ the eigenvalues for $T_\psi$ 

\begin{proposition}
\label{theorem:main-tool} Let $\bar \mu$ an eigenvalue for the operator $T_0$ with
multiplicity $m>1$. Let $U$ and open bounded interval such
that 
\[
\bar{U}\cap\sigma\left(T_0\right)=\left\{ \bar{\mu}\right\} ,
\]
where $\sigma(T_0)$  represents  the spectrum of $T_0$. 

Then, there exists $\psi\in \mathcal{D}$, and $\bar\varepsilon>0$  such that with $\|\psi\|_{\mathcal{D}}\le\bar\varepsilon$ and
\[
\bar{U}\cap\sigma(T_\psi))=\left\{ \mu_{1}^{{\Omega_\psi}},\dots,\mu_{k}^{{\Omega_\psi}}\right\} ,
\]
where $1<k\le m$ and, for all $t=1,\dots, k$, any eigenvalue $\mu_{t}^{{\Omega_\psi}}$ has multiplicity $m_t<m$, , with $\sum_{t=1}^{k}m_{t}=m$.
Also, we can choose $\psi$ such leaves $S$, or, reversely, $W$, untouched.
\end{proposition}
\begin{proof}
We recall that if $\|\psi\|_{\mathcal{D}}$ is small, the multiplicity of an eigenvalue $\mu^{\Omega_\psi}$ near $\bar{\mu}$ is lesser or equal than the multiplicity of $\bar{\mu}$.
In the previous section we have proved that it is not possible to have persistence of multiplicity for small perturbation supported either in $S$ or in $W$. Thus, there exists at least a small perturbation for which the multiplicity decreases.
\end{proof}

The next corollary follows from Proposition \ref{theorem:main-tool}, composing
a finite number of perturbations and recasting the claim for problem \eqref{Slosh} .
\begin{corollary}\label{cor:auto-semplice-0}
Let $\bar\mu$ an eigenvalue for problem \eqref{Slosh} with mutliplicity $m$.  For any $\varepsilon>0$ sufficiently small there exist $\psi \in \mathcal{D}$  with $\|\psi\|_\mathcal{D}<\varepsilon$, which leaves $S$, or  $W$, fixed, for which  Problem (\ref{S-Dpert})  has exactly $m$ simple eigenvalues in a neighborhood of $\bar\mu$ .

\end{corollary}
At this point we are in position to prove the main result of this paper.
\begin{proof}[Proof of Theorem  \ref{th:SD}]
To prove the result, we iterate the procedure of Proposition \ref{theorem:main-tool} and Corollary \ref{cor:auto-semplice-0} countably many times.  
Specifically, suppose that  $ \mu_{q_1}$, for some $q_1\in \mathbb{N}$ is the first multiple eigenvalue, with multiplicity $m_1>1$, for problem \eqref{Slosh} on $\Omega$. Then, by Proposition  \ref{theorem:main-tool} and Corollary  \ref{cor:auto-semplice-0}, there exists $\varepsilon_1$ and a perturbation $\psi_1\in\mathcal{D}$ with  $\|\psi_1\|_\mathcal{D}<\varepsilon_1$, such that for Problem (\ref{S-Dpert}) in $\Omega_{1}:=(I+\psi_1)\Omega$ all the eigenvalues $\mu^{\Omega_1}_1, \mu^{\Omega_1}_1\dots,\mu^{\Omega_1}_{q_1},\dots \mu^{\Omega_1}_{{q_1}+m_1-1}$ are simple. Let us set $F_1:=(I+\psi_1)$. 

If all the eigenvalues of Problem (\ref{S-Dpert}) on $\Omega_1$ are simple, the proof is complete. Otherwise, let $ \mu^{\Omega_1}_{q_2}$, with $q_2\ge q_1+m_1$ be the first multiple eigenvalue with multiplicity $m_2>1$.  In this case, we can find another perturbation $\psi_2$, with $\|\psi_2\|_\mathcal{D}<\varepsilon_2< \varepsilon_1/2$ which splits the eigenvalue  $ \mu^{\psi_1}_{q_2}$ while leaving the multiplicity of the previous eigenvalues unchanged. More precisely, setting $F_2:=(I+\psi_2)$, we define the new perturbed domain $\Omega_2=F_2(\Omega_1)=F_2\circ F_1 (\Omega)$ for which all the eigenvalues $\mu^{\Omega_2}_1,\dots, \mu^{\Omega_2}_{{q_2}+m_2-1}$  for Problem (\ref{S-Dpert}) in $\Omega_2$ are simple. 

We can iterate this procedure countably many times, yielding a sequence of positive numbers $0<\varepsilon_\ell<    \frac{\varepsilon_\ell}{2^\ell}$, a sequence of perturbations $\psi_\ell\in \mathcal{D}$ with $\|\psi_\ell\|_\mathcal{D}<\varepsilon_\ell$, a sequence of maps $F_\ell:=(I+\psi_\ell)$ and a sequence of domains $\psi_\ell=F_\ell\circ\dots\circ F_1(\Omega)$ for  which all the eigenvalues $\mu^{\Omega_\ell},\dots,\mu^{\Omega_\ell}_{q_\ell+m_\ell-1}$  for Problem (\ref{S-Dpert}) in $\Omega_\ell$ are simple. 

To conclude the proof, define $\mathcal{F}_\ell:=F_\ell\circ\dots\circ F_1$. By construction, and due to the choice on the size of $\varepsilon_l$, we have that $\mathcal{F}_\ell\rightarrow \mathcal{F}_\infty$ in $C^2(\mathbb{R}^n,\mathbb{R}^n)$, that $\psi_\infty:=\mathcal{F}_\infty-I$ belongs to $\mathcal{D}$ and that $\|\psi_\infty\|_\mathcal{D}<2\varepsilon_1$. At this point one can easily prove that  the eigenvalues for Problem (\ref{S-Dpert}) in $\Omega_\infty:=I+\psi_\infty(\Omega)$ are all simple and we conclude the proof. \end{proof}

\section{Proof of Theorem \ref{th:SN}}\label{sec:sloshsimp}

We recast problem (\ref{Slosh0}) adding $u$ on both side of the boundary condition on $S$, namely 
\begin{equation*}
   \left\{
    \begin{array}{cc}
         -\Delta u=0& \text{ on } \Omega\\
        \partial_\nu u+u=(\lambda+1) u & \text{ in } S\\
        \partial_\nu u=0 & \text{ in }  W
    \end{array}.
    \right.
\end{equation*}
and, recalling that $d(x):=1_S$, i.e. $d(x)=1$ if $x\in S$ and $d(x)=0$ if $x\in W$, and setting $\hat \lambda:=\lambda+1$, we write  (\ref{Slosh0})  in the more compact form
\begin{equation}\label{Slosh0comp}
   \left\{
    \begin{array}{cc}
         -\Delta u=0& \text{ on } \Omega\\
        \partial_\nu u+d(x)u=\hat \lambda d(x) u & \text{ in } \partial \Omega
    \end{array}.
    \right.
\end{equation}
Consider on $H^1(\Omega)$ the scalar product
\[
\hat a_{\Omega}(u,v):=\int_\Omega \nabla u \nabla v \ dx +\int_{\partial \Omega} duv\ d\sigma .
\]
This scalar product is equivalent to the usual one and therefore equips $H^1(\Omega)$ with the equivalent norm $\| u\|_{H^1}=\|u\|:=\big(a_{\Omega}(u,v)\big)^{1/2}$. As in the previous case, a function $e$ is an eigenfunction with corresponding eigenvalue $\hat\lambda$, that is $e$ solves (\ref{Slosh0comp})  in a weak sense, if and only if 
$$
    \hat a_{\Omega}(e,\varphi)= \hat\lambda\int_{\partial \Omega} d(x)e\varphi\ d\sigma =(\lambda+1)\int_{\partial \Omega} d(x)e\varphi\ d\sigma .
$$
for all $\varphi\in H^1(\Omega)$. 

Henceforth, with a slight abuse of notation, we will retain the notations introduced in Section \ref{Sec:frame}. In particular, we will use the same names for operators and quantities that, although differing in their precise definitions, play analogous roles in the proofs. This choice  is made to to maintain uniformity throughout the manuscript and, in our view, contributes to a clearer presentation of the main results.

Given $f\in H^1(\Omega)$, we consider the continuous linear functional on $H^1(\Omega)$ given by
$$
L_f(\varphi):=\int_{\partial\Omega} d(x)f\varphi\  d\sigma,
$$
and, by Riesz theorem, for any $f\in H^1(\Omega)$, there exists a unique $w\in H^1(\Omega)$, such that 
\[
\hat a_{\Omega}(w,\varphi)=L_f (\varphi)=\int_{\partial \Omega} d(x)f\varphi\  d\sigma .
\]
Again, let $E_\Omega$  the operator, from $ H^1(\Omega)$ in itself, which associates $f$ to ,$w$, i.e. 
\begin{equation}\label{eq:Eusimp}
\hat  a_{\Omega}(w,\varphi)=a_{\Omega}(E_\Omega(f),\varphi)=\int_{\partial \Omega} d(x)f\varphi\  d\sigma 
\end{equation}
for all $\varphi \in H^1 (\Omega)$. As before,  $(\hat \lambda, e)$ are an eigenpair of problem (\ref{Slosh0comp}), in a weak sense, if and only if $e=E_\Omega(\hat \lambda e)$, or, equivalently, $e$ is an eigenfunction of $E_\Omega$ with eigenvalue $\mu=1/\hat \lambda$.

Also, we define the pullback and the pushforward maps $\gamma_\psi:H^1(\Omega_{\psi})\rightarrow H^1(\Omega)$ and $\gamma_ \chi=\gamma_\psi^{-1}$ as introduced in   (\ref{gamma-psi}).

Finally, we pull back the bilinear form $\hat a_{\Omega_\psi} $ defined  on $H^1 (\Omega_\psi)$ as follows. For any $\tilde u,\tilde \varphi\in H^1 (\Omega_\psi)$ set
$$
\begin{aligned}
    \hat a_{\Omega_\psi}(\tilde u,\tilde \varphi)&=
   \int_{ \Omega_\psi}\nabla \tilde u\nabla\tilde \varphi\  dx
    +\int_{\partial \Omega_\psi}d_\psi \tilde u\tilde \varphi\ d\sigma 
    \\ &
    =\mathcal{A}_\psi (u,v)
    +\int_{S} u \varphi B_\psi d\sigma =:\mathcal{E}_\psi (u,\varphi) 
\end{aligned}$$
 where $u=\gamma_{\psi}\tilde u$, $\varphi=\gamma_{\psi} \tilde \varphi$  and $\mathcal{A}_\psi$ and $B_\psi$ are defined, respectively, in Lemma \ref{lem:Astorto} and in Lemma \ref{lem:B}.   Also here $d_\psi(\xi)=1_{S_\psi}=d_\psi(x+\psi(x))=d(x)$ for $\xi:=(x+\psi(x))\in\partial\Omega_\psi$ and for $x\in \partial \Omega$.
We have
\begin{equation}\label{eq:fond2simp}
      \hat  a_{\Omega_\psi}(E_{\Omega_\psi}(\tilde u),\tilde \varphi)= 
     \hat   a_{\Omega_\psi}( E_{\Omega_\psi}(\gamma_\psi^{-1} u),\gamma_\psi^{-1}  \varphi)
=\mathcal{E}_\psi (\gamma_\psi E_{\Omega_\psi}(\gamma_\psi^{-1} u),\varphi)
\end{equation}
so, defining $T_\psi:= \gamma_\psi E_{\Omega_\psi}\gamma_\psi^{-1} $ , $T_\psi :H^1(\Omega)\rightarrow H^1(\Omega)$ we have that $T_0=E_\Omega $,  and, combining (\ref{eq:Eusimp}) and (\ref{eq:fond2simp}), it holds
\begin{equation}\label{fondsimp}
    \mathcal{E}_\psi (T_\psi u,\varphi)=\int_{\partial \Omega_\psi} d_\psi(\xi)\tilde u \tilde \varphi\ B_\psi d\sigma =\int_{\partial \Omega} d(x)u\varphi\ B_\psi d\sigma .
\end{equation} 
By the lemmas in section \ref{Sec:der} we now possess all necessary components  equation (\ref{fondsimp}) with respect to $\psi$, obtaining the principal result of this section
 \begin{corollary}
The operators  $\mathcal{E}_\psi$ and $T_\psi$ are Frechét differentiable 
with respect to $\psi$ at point $0$, and it holds 
\begin{equation*}
    \mathcal{E}_\psi'(0)[\psi] (T_0 u,\varphi)+
      \mathcal{E}_0 (T_\psi '(0)[\psi] u,\varphi)=\left. \frac{\partial}{\partial \psi}\int_{\partial \Omega}d(x)  u\varphi\ B_\psi d\sigma\right|_{\psi=0}
\end{equation*} 
which, in light of Lemma \ref{lem:Astorto} and Lemma \ref{lem:B} becomes
\begin{equation}\begin{aligned}\label{cond1}
   \hat a_\Omega  (T_\psi '(0)[\psi] u,\varphi)&=\left. \frac{\partial}{\partial \psi}\int_{\partial \Omega}d(x)  u\varphi\ B_\psi d\sigma\right|_{\psi=0}
- \mathcal{E}_\psi'(0)[\psi] (T_0 u,\varphi)
\\
&=\left. \frac{\partial}{\partial \psi}\int_{\partial \Omega}d(x)  u\varphi\ B_\psi d\sigma\right|_{\psi=0}\\ &\quad
- \mathcal{A}_\psi'(0)[\psi] (T_0 u,\varphi)
-\left. \frac{\partial}{\partial \psi}\int_{\partial \Omega}d(x) T_0 u\varphi\ B_\psi d\sigma\right|_{\psi=0}
\\ & =-L_1( T_0u, \varphi)-
I_1( T_0u, \varphi)
-G(T_0 u, \varphi)+G(u,\varphi).
\end{aligned}\end{equation}
\end{corollary}

As in the preceding section, let $\mu=1/\hat \lambda=1/(\lambda +1)$ be an eigenvalue of $T_0$ with multiplicity $m>1$, and let $\{e_1, \dots,e_m \}$ be an orthonormal base for the corresponding  eigenspace. By the no splitting condition, if $\psi$ preserves the multiplicity of $\mu$, then there exists $\rho=\rho(\psi)\in \mathbb{R}$ such that
$$
    \hat   a_\Omega  (T_\psi '(0)[\psi] e_r,e_s)=  \rho(\psi)\delta_{rs}.
$$
Having in mind (\ref{cond1}), and using that $T_0e_r=\mu e_r$ for all $r=1,\dots,m$, it follows that
\begin{equation}
    \label{cond2}
    -L_1(\mu e_r,e_s)-
I_1(\mu  e_r,e_s)
-G(\mu  e_r, e_s)+G( e_r,e_s)=\rho\delta_{rs}.
\end{equation}
Choosing a perturbation $\psi$ compactly supported in the interior of $\Omega$ we can, as previously shown, absorb $L_1$  in the right hand side of (\ref{cond2}). Hence, the no splitting condition becomes
 \begin{equation}{\label{nosplitsimp}}
    -I_1( \mu e_r,e_s)
-G( \mu e_r, e_s)+G(e_r,e_s)=C(\psi)\delta_{rs}.
\end{equation}

\subsection{Perturbations supported in $S$}\label{sec:Sneumann}

Now, we take a perturbation which leaves $W$ fixed. The no splitting condition simplifies as
$$
    -I_1(  \mu e_r,e_s)+(1-\mu) G( e_r, e_s)
=C(\psi)\delta_{rs}, 
$$
that is
 \begin{equation} \begin{aligned}\label{eq:splitS0}
 C(\psi)\delta_{rs}  &= - \mu\int_{S}\psi_t\nu_t\nabla e_r\nabla e_s d\sigma
    + \mu \int_{S}
\psi_i\nu_j \left(\frac{\partial e_r}{\partial x_i}\frac{\partial e_s}{\partial x_j}+\frac{\partial e_r}{\partial x_j}\frac{\partial e_s}{\partial x_i}
\right) d\sigma\\ &\quad
+(1-\mu)  \int_{S}e_r e_s \left(\div \psi -\sum_{r=1}^n \partial_\nu \psi_r \nu_r\right)d\sigma.\end{aligned}
 \end{equation}
With a strategy similar to the one performed in \ref{SDpertS}, we notice that  $\nu_j\frac{\partial e_r}{\partial x_j}=\partial_\nu e_r=\lambda e_r=(\hat \lambda-1) e_r$, so integrating by parts we get, for the second term of (\ref{eq:splitS0}), 
$$\begin{aligned}
    \mu \int_{S}
\psi_i\nu_j \left(\frac{\partial e_r}{\partial x_i}\frac{\partial e_s}{\partial x_j}+\frac{\partial e_r}{\partial x_j}\frac{\partial e_s}{\partial x_i}
\right) d\sigma&=
\mu(\hat \lambda-1) \int_{S}
\psi_i\frac{\partial }{\partial x_i}\left(e_r e_s
\right) d\sigma\\ &=-\frac{ \hat \lambda-1} {\hat \lambda }
\int_{S}
e_r e_s\div\psi
 d\sigma,
\end{aligned}$$
and, since in the third integral of (\ref{eq:splitS0}), $(1-\mu)=( \hat \lambda-1)/{ \hat \lambda}$, the no splitting condition becomes

$$     \frac 1{ \hat \lambda}\int_{S}\psi_t\nu_t\nabla e_r\nabla e_s d\sigma
+\frac{ \hat \lambda-1}{ \hat \lambda}   \int_{S}e_r e_s \left(\sum_{r=1}^n \partial_\nu \psi_r \nu_r\right)d\sigma \ =
C(\psi)\delta_{rs}.
$$
There exists at least one component $\nu_t \neq 0$. For the sake of simplicity, let us take $\nu_1\neq 0$.
Now choose $\psi=(\psi_1,0,\dots,0)$ with $\psi_1$ arbitrarily defined on $S$ and such that $\frac{\partial \psi_1} {\partial \nu}=0$. Thus  $\nabla e_r \nabla e_s=0$ if $r\neq s$ and $|\nabla e_1|^2=\dots=|\nabla e_m|^2$ almost everywhere on $S$ and the term $\int_{S}\psi_t\nu_t \nabla e_r\nabla e_s d\sigma$ can be absorbed in the left hand side of the condition and we remain with
$$
     C(\psi)\delta_{rs}=
     \frac{ \hat \lambda-1}{ \hat \lambda} \int_{ S}  e_r  e_s\sum_{r=1}^n \partial_\nu \psi_r \nu_r d\sigma
$$
Since we can choose $\sum_{r=1}^n \partial_\nu \psi_r \nu_r $ as and arbitrary function on $S$, we get that $ e_r  e_s=0$ if $r\neq s$ and $| e_1|^2=\dots=| e_m|^2$ almost everywhere on $S$. This implies that on $S$ we get, for any $r=1,\dots, m$
$$
    e_r=0\ \hbox{and}\ \frac{\partial e_r}{\partial \nu}=( \hat \lambda-1) e_r=0
,$$
which leads to a contradiction.

\subsection{Perturbations supported in $W$ (when $\Omega \subset \mathbb{R}^2$)}
If $\psi$ leaves $S$ fixed, namely,  the support of $\psi$ intersects $\partial \Omega$ only within the set $W$, we have that $G( e_r,e_s)=G(T_0 e_r,e_s)\equiv 0$. Proceeding as in Subsection \ref{SDpertW}, we get that the no splitting condition (\ref{nosplitsimp}) becomes
\begin{equation}\label{eq:nospiltDir}
 C(\psi)\delta_{rs}=-\mu\int_{W}\psi_t \nu_t \nabla e_r\nabla e_s
 d\sigma
+\mu\int_{W}
\psi_t\left(\frac{\partial e_r}{\partial \nu}\frac{\partial e_s}{\partial x_t}+\frac{\partial e_r}{\partial x_t}\frac{\partial e_s}{\partial \nu}
\right) d\sigma 
\end{equation}
Now, since $\partial_\nu e_r=\partial_\nu e_s=0$, on $W$, we have $\left(\frac{\partial e_r}{\partial \nu}\frac{\partial e_s}{\partial x_t}+\frac{\partial e_r}{\partial x_t}\frac{\partial e_s}{\partial \nu}
\right)=0$, so
\begin{equation*}
    \mu\int_{W}\psi_t \nu_t \nabla e_r\nabla e_sd\sigma
=C(\psi)\delta_{rs}.
\end{equation*}
By the arbitrariness of  $\psi$ we have that if $r\neq s $ then $\nabla e_r\nabla e_s=0$ almost everywhere on $W$, and that $|\nabla e_1|^2=\dots=|\nabla e_m|^2$ almost everywhere on $W$. Unfortunately, we were not able to derive a contradiction based on this condition in general. However,  we can complete the proof of the Theorem \ref{th:SN} in the case $\Omega\subset \mathbb{R}^2$. 
Since $\partial_\nu e_r=0$ on $W$ for $r=1,\dots,m$,  consider $\tau$ the tangent vector in a point to $W$.  At this point we have that, by the no splitting condition, the set $\{\partial_\tau e_1/|\partial_\tau e_1|, \dots \partial_\tau e_m/|\partial_\tau e_m| \}$ is an orthonormal set in the tangent space to $W$ which is  $\mathbb{R}$.  This is impossible unless $m=1$ and this ends the proof.

\begin{remark}\label{rem-n3}
    We want to point out that this latter idea partially works in higher dimension $n\ge 3$. In this case the set of tangential derivatives is an orthonormal set in $\mathbb{R}^{n-1}$, so it is possible to split any eigenvalue with a perturbation which leaves $W$ fixed, up to a multiplicity of $n-1$. 
\end{remark}

\appendix

\section{A different variational approach to Theorem  \ref{th:SN}.}\label{Sec:frame-bis}
In this section we give a sketch of an alternative proof of Theorem  \ref{th:SN}, 
which could be interesting per se, since the variational framework and the pull back and push back forms are different.
The main point in the technique of the proof is the choice of a good variational space with a scalar product that make the weak formulation easy.
One possibility, when attacking problem (\ref{Slosh0}), is to work in the Hilbert space $\mathcal{H}(\Omega):=\{u\in H^1(\Omega)\ :\ \int_{\partial \Omega}du=0=\int_S u \}$ equipped with the equivalent scalar product
$$a_\Omega(u,v):=\int_\Omega \nabla u \nabla v dx.$$
One can show easily that $u$ is an eigenfunction with eigenvalue $\lambda$,   in a weak sense if and only if 
\begin{equation}\label{eq:weakSD}
    a_{\Omega}(u,\varphi)=\lambda\int_{\partial \Omega} u\varphi\ d\sigma  \text{ for all }\varphi\in \mathcal{H}(\Omega)
\end{equation}

and that it is easy to show that any $u\in H^1(\Omega)$ eigenfunction of  (\ref{Slosh0}) with eigenvalue $\lambda\neq0$  belongs to $ \mathcal{H}(\Omega)$.

As ususal we definte $E_\Omega (f)$, as the operator from  $ \mathcal{H}(\Omega)$ in itself such that
\begin{equation}
 a_\Omega(E_\Omega(f),\varphi)=\int_{S} f\varphi\  d\sigma 
\text{ for all }\varphi \in \mathcal{H},
\end{equation}
 so, $(\lambda, u)$, is an eigenpair of (\ref{Slosh0}), in a weak sense, if and only if $u=E_\Omega(\lambda u)$, or equivalently, if $E_\Omega (u)=\frac 1\lambda u$.   

Now we consider the operator $E_{\Omega_\psi}$ on the perturbed domain $\Omega_\psi$ and we have to pull it back in the space $\mathcal{H}(\Omega)$. 

To this end the usual map 
\begin{align*}
\gamma_\psi&:{H^1}(\Omega_{\psi})\rightarrow {H^1}(\Omega);\\
\gamma_\psi(\tilde{u})&:=\tilde u\circ (I+\psi)(x).
\end{align*}
is not sufficient, since we have to deal with a change of variable in the condition $\int_{\partial \Omega}du=0$ in the definition of the space $\mathcal{H}(\Omega)$.

Called, as before  $B_\psi(x)$, defined on $\partial \Omega$, be   the Jacobian  of the boundary change of variables from $\partial\Omega_\psi $ to $\partial\Omega$, induced by $\psi$., we can extend this function $B_\psi$ to the whole $\Omega$ as follows. 
Let $\Sigma_k=\{x\in \Omega \ :\ d(x,\partial \Omega)<k\}$. Since $\Omega$ is regular there exist $k$ such that $\Sigma_k$ is a tubular neighborhood. Then, for $x\in \Sigma_k$ we can write $x=y+t\nu(y)$ where $y\in\partial\Omega$ and $t>0$, and we define
\begin{equation}
    \bar{B}_\psi(x)=\left\{
    \begin{array}{cc}
    1&x\in \Omega\smallsetminus \Sigma_k\\
  \left(1-\frac tk\right)B_\psi(y)+\frac tk  &x\in \Sigma_k
    \end{array}
    \right. .
\end{equation}
Also, since $\left. \frac{\partial}{\partial \psi}B_\psi \right|_{\psi=0}=\div \psi -  \partial_\nu \psi_r \nu_r$, set $\beta_\psi(y)=\div \psi -  \partial_\nu \psi_r\nu_r$ we can write ${B}_\psi(y)=1+\beta(y)+o(|\psi|)$ and
\begin{equation}
   \bar{B}_\psi(x)=\left\{
    \begin{array}{cc}
    1&x\in \Omega\smallsetminus \Sigma_k\\
  1+\left(1-\frac tk\right)\beta_\psi(y)+o(|\psi|)  &x\in \Sigma_k
    \end{array}
    \right. .
\end{equation}
When no ambiguity arises, we simply write ${B}_\psi(x)$ instead of $ \bar{B}_\psi(x)$.

Now we can define the correct pull back map by means of $\gamma_\psi$ and  $B_\psi$ as follows.
\begin{align*}
\hat\gamma_\psi&:\mathcal{H}(\Omega_\psi)\rightarrow \mathcal{H}(\Omega);\\
\hat\gamma_\psi(\tilde{u})&:=\gamma_\psi(\tilde{u})B_\psi(x)=u(x)B_\psi(x)
\end{align*}

At this point, given $\tilde u,\tilde \varphi \in \mathcal{H}(\Omega_\psi)$ we define a bilinear form $\hat{\mathcal{A}}_\psi$  on $\mathcal{H}(\Omega)$ by pulling back by $\hat \gamma_\psi$ the bilinear form  $a_{\Omega_\psi} $ on $\mathcal{H} (\Omega_\psi)$ as follows:
\begin{equation}
a_{\Omega_\psi}(\tilde u,\tilde \varphi)=\int_{\Omega_\psi} \nabla \tilde u \nabla \tilde \varphi \ dx  =:\hat{\mathcal{A}}_\psi (\hat u,\hat \varphi) 
\end{equation}
where $\hat u=\hat \gamma_{\psi}\tilde u$ and $\hat\varphi=\hat\gamma_{\psi} \tilde \varphi$  and, as before, we define
 \begin{equation}
 T_\psi({\cdot }):=\hat\gamma_\psi E_{\Omega_\psi} \hat\gamma_\psi^{-1}({\cdot}).
 \end{equation}
The price we pay with this approach is that the identity which is crucial to find $T'_\psi(0)$ becomes
  \begin{multline}
       \label{eq:fond-bis}
     \hat{\mathcal{A}}_\psi (T_\psi{\hat u},\hat \varphi) =a_{\Omega_\psi}(E_{\psi}(\tilde u),\tilde \varphi)=\int_{S_\psi}  \tilde u \tilde \varphi\  d\sigma \\
     =\int_{S} u\varphi\ B_\psi d\sigma 
      =\int_{S} \hat u\hat \varphi\ \frac1{B_\psi} d\sigma 
 \    \text { for all }\varphi \in \mathcal{H}(\Omega). 
 \end{multline}

By the following Taylor expansion of $1/B_\psi$
\begin{multline}
    \label{derBpsi-bis}
    B_\psi^{-1}=1-\left(1-\frac tk\right)\beta_\psi(y)\mathbbm{1}_{\Sigma_k}+o(|\psi|)=\\
    =1-\left(1-\frac tk\right)\left(\div \psi -  \partial_\nu \psi_r\nu_r\right)\mathbbm{1}_{\Sigma_k}+o(|\psi|)
\end{multline}

we can differentiate (\ref{eq:fond-bis}), obtaining 
\begin{equation}\label{nosplit-bis}
a_\Omega(T_\psi '(0)[\psi]e_r,e_s)=
-\mu L_1(e_r,e_s)
-\mu I_1(e_r,e_s)+G(e_r,e_s).
\end{equation}
which is completely analogous to (\ref{cond1}), giving rise to the same no splitting condition. At this point the proof of Theorem \ref{th:SN} follows identically.

\bibliography{references}
\bibliographystyle{amsplain}

\end{document}